\documentclass[a4paper,11pt,reqno]{amsart}

\usepackage{amsmath}%
\usepackage{amsfonts}%
\usepackage{amssymb}%
\usepackage{graphicx}
\usepackage{float}
\usepackage{amssymb}
\usepackage{color}
\usepackage{dsfont}
\usepackage{setspace}
\usepackage{cite}

\theoremstyle{plain}

\newtheorem{theorem}{Theorem}
\newtheorem{definition}[theorem]{Definition}
\newtheorem{example}[theorem]{Example}

\newtheorem{proposition}[theorem]{Proposition}
\newtheorem{remark}[theorem]{Remark}

\allowdisplaybreaks

\def\N{\mathbb{N}}

\def\Q{\mathbb{Q}}
\def\R{\mathbb{R}}

\def\P{\mathcal{P}}
\def\Oo{\mathcal{O}}
\def\ex{\mathbb{E}}

\def\eins{\mathds{1}}
\def\lin{{\rm lin}}
\def\ve{\varepsilon}

\numberwithin{equation}{section}

\begin{document}

\title[Optimal approximation of Skorohod integrals - examples with substandard rates]{Optimal approximation of Skorohod integrals - examples with substandard rates}\label{}

\author[P. Parczewski]{Peter Parczewski}


\address{University of Mannheim, Institute of Mathematics A5, 6, D-68131 Mannheim, Germany.}
\email{parczewski@math.uni-mannheim.de}

\date{\today}

\begin{abstract}
We consider optimal approximation with respect to the mean square error of It\^o integrals and Skorohod integrals given an equidistant discretization of the Brownian motion. We obtain for suitable integrands optimal rates smaller than the standard $n^{-1}$, where $n$ denotes the number of evaluations of the Brownian motion. For the It\^o integral this is due to the Weyl equidistribution theorem and discontinuities of the integrand. For the Skorohod integral the situation is more complicated and relies on a reformulation of the Wiener chaos expansion. Here, we specify conditions on the integrands to obtain optimal rates $n^{-1/2}$, respectively, examples of lower rates.
\end{abstract}



\maketitle

\section{Introduction}

Suppose a Brownian motion $(W_{t})_{t \in [0,1]}$ on the probability space $(\Omega, \mathcal{F}, P)$, where the $\sigma$-field $\mathcal{F}$ is generated by the Brownian motion and completed by null sets. Our aim in this note is to investigate the mean square error
\begin{align*}
e_{n} := \ex\left[\left(\int_{0}^{1} u_s dW_s - \ex\left[\left.\int_{0}^{1} u_s dW_s\right|W_{\frac{1}{n}}, W_{\frac{2}{n}}, \ldots, W_1 \right]\right)^2\right]^{1/2}
\end{align*}
for some It\^o integrals and Skorohod integrals with sufficiently irregular integrands $(u_t)_{t \in [0,1]}$. Our main result is that
\begin{align}\label{eq:intro1}
e_{n_k} \sim c \, n_k^{-1/2} 
\end{align}
for appropriate strictly increasing sequences $(n_k)_{k \in \N}$ and appropriate integrands. The Skorohod integral is a natural extension of the It\^o integral to nonadapted integrands, see e.g. \cite{Nualart, Di_Nunno_und_so}. The asymptotic behaviour in \eqref{eq:intro1} is in contrast to the ordinary optimal approximation error 
\begin{align}\label{eq:intro2}
e_n \sim c \, n^{-1}  
\end{align}
for It\^o integrable and sufficiently smooth integrands $u_s = f(s,W_s)$, see e.g. \cite{Mueller_Gronbach, Przybylowicz, Kloeden_Platen}.


In \cite{NP} we considered simple Skorohod integrands of the type
\begin{align}\label{eq:integrandIntro}
\left(u_{t}= f(t, W_{t}, W_{\tau_2},\ldots, W_{\tau_K})\right)_{t \in [0,1]}  
\end{align}
for some $f \in C([0,1] \times \R^{K}; \R)$ and fixed $\tau_2,\ldots, \tau_K \in (0,1]$. Under some smoothness and growth conditions on $f$ and $\tau_2,\ldots, \tau_K \in \frac{1}{n}\N$, we obtained \eqref{eq:intro2} where the constant $c$ is a natural extension of the constant in the It\^o case. However, treating integrands with the condition $\tau_2,\ldots, \tau_K \in \frac{1}{n}\N$ violated, the situation changes extremely. Starting from the (stationary) Skorohod integrable process
\[
u_t = \sum_{k \geq 0} \frac{1}{k!} h^{k}_{\{kT\}}(W_{\{kT\}}), 
\]
for the Hermite polynomials $h^{k}_{\alpha}(x)$ (see below), the fractional part $\{kT\} := kT - \lfloor kT \rfloor $ and some fixed $T \in (0,1)\setminus \Q$, we are interested into the optimal approximation problem with integrands beyond the ordinary regularity assumptions. For these integrands no finite extension of the equidistant information to $W_{1/n}, \ldots, W_{1}, W_{x_1}, \ldots, W_{x_N}$ suffices to evaluate the corresponding Skorohod integral exactly.

Thus, in contrast to \eqref{eq:intro2}, under some smoothness assumptions on the integrand \eqref{eq:integrandIntro}, we obtain as the main result in Theorem \ref{thm:OptApproxIrrational} a strictly increasing sequence $(n_k)_{k \in \N}$ and the asymptotic behaviour \eqref{eq:intro1} with the new constant
\begin{align*}
c_2 &= \frac{1}{2}  \left(\sum_{i=2}^{K}\ex\left[\left(\int_{0}^{1}u^i_s dW_s\right)^2\right]\right)^{1/2}
\end{align*}
depending on further Skorohod integrals, where the integrands $(u^i_t)_{t \in [0,1]}$ are specified by the initial process $(u_t)_{t\in [0,1]}$ in \eqref{eq:integrandIntro}. 

We observe the same optimal rate \eqref{eq:intro2} for It\^o integrals with discontinuous integrands and specify the constant $c$ (Theorem \ref{theorem:MainTheoremIto}). 

Dealing with stationary integrands of more irregular type as $u_t =|W_T|$ for some fixed $T \in (0,1) \setminus \Q$, we obtain smaller optimal rates, e.g. optimal rates below $n^{-1/2}$.

The recent work \cite{Jentzen_MG_Y} presents It\^o stochastic differential equations with infinitely often differentiable and globally bounded coefficients such that the strong convergence rate of all numerical schemes based on a finite information of the driving Brownian motion is not polynomial.


However, in contrast to that, our results focus on the Skorohod integral. The integrand requires appropriate continuity assumptions to ensure the existence of the Skorohod integral itself and the elements in the constants in \eqref{eq:intro1}.

The paper is organized as follows: In Section \ref{section_1} we give some preliminaries on the Skorohod integral and present some helpful reformulations of the Wiener chaos decomposition. The main results on the optimal approximation of Skorohod integrals of integrands of the type $u_{t}= f(t, W_{t}, W_{\tau_2},\ldots, W_{\tau_K}))_{t \in [0,1]}$
 are the content of Section \ref{section_2}. The proofs are postponed to Section \ref{section_3}. Finally, in Section \ref{section_4} we study more irregular cases with optimal rates below $n^{-1/2}$.

\section{Preliminaries}\label{section_1}

Here, we give a short introduction to the Skorohod integral. An essential tool is the reformulation of processes via the Wiener chaos decomposition. Aiming the optimal approximation, we collect some basic properties of conditional expectations and Skorohod integrals and present a simple It\^o formula for Skorohod integrals. 

We restrict ourselves to the stochastic calculus on the Gaussian Hilbert space $\{I(f) : f \in L^2([0,1])\} \subset L^2(\Omega)$, where $I(f)$ denotes the Wiener integral. We denote the norm and inner product on $L^2([0,1])$ by $\|\cdot \|$ and $\langle \cdot, \cdot\rangle$. Due to the totality of the stochastic exponentials
\begin{equation*}\label{eq:WickExponential}
\exp\left(I(f) - \|f\|^2/2\right)\ , \quad f \in L^2([0,1]), 
\end{equation*}
in $L^2(\Omega)$, (see e.g. \cite[Corollary 3.40]{Janson}), for every $X \in L^2(\Omega, \mathcal{F}, P)$ and $h \in L^2([0,1])$, the \emph{S-transform} of $X$ at $h$ is defined as
\[
(SX)(h) := \ex[X \exp\left(I(f) - \|f\|^2/2\right)]. 
\]
The S-transform $(S \cdot)(\cdot)$ is a continuous and injective function on $L^2(\Omega, \mathcal{F}, P)$ (see e.g. \cite[Chapter 16]{Janson} for more details). As an example, for $f,g \in L^2([0,1])$, we have $(S \ \exp\left(I(f) - \|f\|^2/2\right))(g) = \exp\left(\langle f,g\rangle\right)$. In particular the characterization of random variables via the S-transform can be used to introduce the Skorohod integral, an extension of the It\^{o} integral to nonadapted integrands (cf. e.g. \cite[Section 16.4]{Janson}): 
\begin{definition}
Suppose $u=(u_t)_{t \in [0,1]} \in L^2(\Omega \times [0,1])$ is a (possibly nonadapted) square integrable process on $(\Omega, \mathcal{F}, P)$ and $Y \in L^2(\Omega, \mathcal{F}, P)$ such that
\[
\forall g \in L^2([0,1]): \ (S Y)(g) = \int_{0}^{1} (S u_s)(g) g(s) ds, 
\]
then $\int_{0}^{1}u_s dW_{s} = Y$ defines the Skorohod integral of $u$ with respect to the Brownian motion $W$. 
\end{definition}
For more information on the Skorohod integral we refer to \cite{Janson}, \cite{Kuo} or \cite{Nualart}. 
We recall that for the Hermite polynomials
\begin{equation*} 
h^{k}_{\alpha}(x) = (-\alpha)^{k} \exp\left(\frac{x^2}{2\alpha}\right) \frac{d^k}{dx^k} \exp\left(\frac{- x^2}{2\alpha}\right) 
\end{equation*} 
and every $k \in \N$, the $k$-th Wiener chaos $H^{:k:}$ is the $L^2$-completion of $\{h^k_{\| f \|^2}(I(f)) : f \in L^2([0,1])\}$ in $L^2(\Omega)$ and these subspaces are orthogonal and fulfill $L^2(\Omega, \mathcal{F}, P) = \bigoplus_{k \geq 0} H^{:k:}$. Thus, for the projections 
\[
\pi_k: L^2(\Omega) \rightarrow H^{:k:},
\]
for every random variable $X \in L^2(\Omega)$, we denote the \emph{Wiener chaos decomposition} as
\[
X=\sum_{k=0}^{\infty} \pi_k(X).
\]
We refer to \cite{Janson, Holden_Buch} for further details and a reformulation in terms of multiple Wiener integrals. We recall that a process $u \in L^2(\Omega \times [0,1])$ is Skorohod integrable if and only if
\[
\sum_{k=0}^{\infty} (k+1) \|\pi_k(u)\|^2_{L^2(\Omega \times [0,1])} < \infty,
\]
(cf. \cite[Theorem 7.39]{Janson}).
The S-transform is closely related to a product imitating uncorrelated random variables as $\ex[X \diamond Y]=\ex[X]\ex[Y]$, which is implicitly contained in the Skorohod integral and a fundamental tool in stochastic analysis. Due to the injectivity of the S-transform, the \emph{Wick product} can be introduced via
\[
\forall g \in L^2([0,1]) \ : \ S(X \diamond Y)(g) = (S X)(g) (S Y)(g)
\]
on a dense subset in $L^{2}(\Omega) \times L^{2}(\Omega)$. For more details on Wick product we refer to \cite{Janson, Holden_Buch, Kuo}. For example, for a Gaussian random variable $X$ and all $k \in \N$, we see that Hermite polynomials play the role of monomials in standard calculus as
\[
X^{\diamond k} = h^k_{\ex[X^2]}(X).
\]

Following \cite{Holden_Buch, Buckdahn_Nualart}, we denote the Wiener chaos decomposition in terms of Wick products as the \emph{Wick-analytic representation}. In particular, for fixed $t_1, \ldots, t_K \in [0,1]$, $f: \R^K \rightarrow \R$ and a square integrable left hand side, there exist $a_{l_1, \ldots, l_K} \in \R, l_1, \ldots, l_K\geq 0$, such that
\begin{equation}\label{eq:WienerChaosByWickProducts}
f(W_{t_1}, \ldots, W_{t_K}) = \sum_{l_1, \ldots, l_K\geq 0} a_{l_1, \ldots, l_K} W_{t_1}^{\diamond l_1} \diamond \cdots \diamond W_{t_K}^{\diamond l_K}.
\end{equation}  

For a proof of the following reformulations we refer to \cite[Chapter 16]{Janson} or \cite[Proposition 7]{NP}:

\begin{proposition}\label{prop:SkorohodAndWick}
Suppose $X \in L^2(\Omega)$ and a Skorohod integrable process  $u\in L^2(\Omega \times [0,1])$. Then, if both sides exist in $L^2(\Omega)$:
\[
\int_{0}^{1} X \diamond u_s dW_s = X \diamond \int_{0}^{1} u_s dW_s, \qquad  \int_{0}^{1} X \diamond u_s ds = X \diamond \int_{0}^{1} u_s ds.
\]
\end{proposition}

Suppose $f_1, \ldots, f_K \in L^2([0,1])$ and $l_1,\ldots, l_K \in \N$. Then the polynomial
\[
h(x_1,\ldots, x_K) =  \dfrac{\partial^{\sum l_i}}{\partial a_1^{l_1} \cdots \partial a_K^{l_K}} \left.\exp\left(\sum a_i x_i - \|\sum a_i f_i\|^2/2\right)\right|_{a_1=\ldots =a_K=0}, 
\]
(where we omit the dependence on $f_i$ and $l_i$ for shortance) gives the analytic representation
\begin{equation}\label{eq:WickProductByWickExponential}
h(I(f_1),\ldots, I(f_K)) = I(f_1)^{\diamond l_1} \diamond \cdots \diamond I(f_K)^{\diamond l_K},
\end{equation}
(cf. \cite[Chapter 3]{Janson}). For more details on these multivariate generalizations of Hermite polynomials see e.g. \cite{Avram_Taqqu} or \cite{P2}. Dealing with $L^2$-norms of Gaussian random variables, we will frequently make use of
\begin{equation}\label{eq:NormOFWickProductGaussian}
\ex\left[\left(I(f_1) \diamond \cdots \diamond I(f_n)\right) \left(I(g_1) \diamond \cdots \diamond I(g_m)\right)\right] = \delta_{n,m} \sum\limits_{\sigma \in \mathcal{S}_{n}} \prod\limits_{i=1}^{n} \langle f_i, g_{\sigma(i)}\rangle,
\end{equation}
for all $n,m \in \N$, $f_1,\ldots, f_n, g_1,\ldots, g_m \in L^2([0,1])$, where $\mathcal{S}_{n}$ denotes the group of permutations on $\{1, \ldots, n\}$ (see e.g. \cite[Theorem 3.9]{Janson}). In particular
\begin{equation}\label{eq:L2NormWickProductInequality}
\ex\left[\left(I(f_1)^{\diamond l_1} \diamond \cdots I(f_n)^{\diamond l_n}\right)^2\right] \leq n! \prod_{i=1}^n l_i! \|f_i\|^{2 l_i}.
\end{equation}
We denote the equidistant discretization of the underlying Brownian motion by
$$
\P_n:= \{W_{\frac{1}{n}}, W_{\frac{2}{n}}, \ldots, W_1\}
$$ 
and the corresponding linear interpolation of the Brownian motion as
$$
W_t^{\lin} := W_{i/n} + n(t-i/n)(W_{(i+1)/n}-W_{i/n}),
$$
for $t \in [i/n, (i+1)/n)$, $i =0,\ldots, n-1$. Obviously, we have $W_t^{\lin} = \ex[W_t|\P_n]$. The advantage of the Wick product is preserving the conditional expectation (see \cite[Corollary 9.4]{Janson}):

\begin{proposition}\label{prop:WickConditionalExpectation}
For $X,Y, X\diamond Y \in L^2(\Omega)$ and the sub-$\sigma$-field $\mathcal{G} \subseteq \mathcal{F}$:
$$
\ex[X\diamond Y|\mathcal{G}] = \ex[X|\mathcal{G}]\diamond \ex[Y|\mathcal{G}].
$$ 
\end{proposition}

\begin{definition}\label{def_Wick_analytic_functionals}
We define the class of \emph{Wick-analytic functionals} as 
\begin{equation}\label{eq:WickAnalyticFunctional}
F^{\diamond}(W_{t_1}, \ldots, W_{t_K}) = \sum\limits_{k=0}^{\infty}a_{1,k} W_{t_1}^{\diamond k} \diamond \cdots \diamond \sum\limits_{k=0}^{\infty} a_{K,k} W_{t_K}^{\diamond k}, \ \  \max_{i \leq K}\sup\limits_{k} \sqrt[k]{k! |a_{i,k}|}  < \infty. 
\end{equation}
\end{definition}
In particular, the analytic representation 
$$
f^{\diamond}(t_1,\ldots, t_K, W_{t_1}, \ldots, W_{t_K})
$$
of a Wick-analytic functional \eqref{eq:WickAnalyticFunctional} via \eqref{eq:WickProductByWickExponential} fulfills $f^{\diamond} \in C^{\infty}([0,1]^K\times \R^{K},\R)$ and all derivatives of Wick-analytic functionals are Wick-analytic as well (cf. \cite[Proposition 10]{NP}).

Combining the previous propositions, we obtained the following reformulation (Remark 11 in \cite{NP}):

\begin{proposition}\label{prop:AnalyticRepresentation}
Suppose $f \in C^{1;2,\ldots, 2}([0,1]\times \R^{K})$ and fixed  $\tau_2, \ldots, \tau_K \in [0,1]$. Then, via \eqref{eq:WienerChaosByWickProducts} and \eqref{eq:WickProductByWickExponential}, the analytic representation
\[
f^{\diamond}(t, \tau_2, \ldots, \tau_K,W_{t}, W_{\tau_2},\ldots, W_{\tau_K}) = f(t,W_{t}, W_{\tau_2},\ldots, W_{\tau_K})
\]
satisfies $f^{\diamond} \in C^{1,\ldots, 1; 2,\ldots, 2}([0,1]^K\times \R^{K})$. 
\end{proposition}

This yields a simple Skorohod It\^o formula (\cite[Theorem 15]{NP}):

\begin{theorem}\label{theorem:SkorohodItoFormula}
Suppose $K \in \N$, $f \in C^{1;2,\ldots, 2}([0,1]\times \R^{K})$ and fixed  $\tau_2, \ldots, \tau_K \in [0,1]$. Then $\left(\dfrac{\partial}{\partial x_1}f(t,W_{t}, W_{\tau_2},\ldots, W_{\tau_K})\right)_{t \in [0,1]}$ is Skorohod integrable and 
\begin{align}
&f(1,W_{1}, W_{\tau_2},\ldots, W_{\tau_K}) - f(0,W_{0}, W_{\tau_2},\ldots, W_{\tau_K})\nonumber\\ &=\int_{0}^{1}\dfrac{\partial}{\partial x_1} f^{\diamond}(t, \tau_2\ldots, \tau_K,W_{t}, \ldots, W_{\tau_K}) dW_{t} + \int_{0}^{1} \mathcal{L} f^{\diamond}(t, \tau_2\ldots, \tau_K,W_{t}, \ldots, W_{\tau_K})dt,\label{eq_thm_Ito_formula_Skorohod}
\end{align}
with $f^{\diamond} \in C^{1,\ldots, 1; 2,\ldots, 2}([0,1]^K\times \R^{K})$ from Proposition \ref{prop:AnalyticRepresentation} and the differential operator $\mathcal{L}$ is defined as
\begin{equation}\label{eq_L_operator}
\mathcal{L}:= \left(\sum_{1\leq k \leq K}\dfrac{\partial}{\partial t_k} + \frac{1}{2} \sum_{1\leq k,l \leq K} \dfrac{\partial^2}{\partial x_k \partial x_l}\right).  
\end{equation}
\end{theorem}

It is mainly due to the following integration by parts formula which can be checked by S-transform and will be useful below:

\begin{align}
&a_{l_1, \ldots, l_K}(1) W_1^{\diamond l_1} \diamond \cdots \diamond W_{\tau_K}^{\diamond l_K} - a_{l_1, \ldots, l_K}(0) W_0^{\diamond l_1} \diamond \cdots \diamond W_{\tau_K}^{\diamond l_K}\nonumber\\
&= \int_{0}^{1} l_1a_{l_1, \ldots, l_K}(t) W_{t}^{\diamond l_1-1} \diamond \cdots \diamond W_{\tau_K}^{\diamond l_K} dW_{t}  + \int_{0}^{1} \left( a_{l_1, \ldots, l_K}'(t)\right) W_{t}^{\diamond l_1} \diamond \cdots \diamond W_{\tau_K}^{\diamond l_K} dt,\label{eq:IBPWick}
\end{align}
where $a_{l_1, \ldots, l_K}(\cdot)\in C^1([0,1])$ for all $l_1, \ldots, l_K \geq 0$.

\section{The main result}\label{section_2}

In \cite{NP} we specified some sufficient and equivalent assumptions for Skorohod integrands to ensure exact simulation, i.e. $e_n=0$. Among them we had (Theorem 17):
\begin{enumerate}
 \item There exists a Wick-analytic representation of $(u_t)_{t \in [0,1]}$ as in Definition \ref{def_Wick_analytic_functionals}.
\item The analytic representation fulfills $f^{\diamond} \in C^{1,\ldots, 1; 2,\ldots, 2}([0,1]^K\times \R^{K})$ and $\mathcal{L} f^{\diamond} =0$ on $[0,1]^K \times \R^{K}$.
\end{enumerate}

This motivated the following optimal approximation result which covers the ordinary It\^o case (\cite[Theorem 21]{NP}):

\begin{theorem}\label{theorem:OptimalApproxSkorohod}
Suppose $K\in \N$, $f^{\diamond} \in C^{1,\ldots, 1; 2,\ldots, 2}([0,1]^K\times \R^{K})$ with
some linear and H\"older growth conditions on $\mathcal{L} f^{\diamond}$, $\tau_2, \ldots, \tau_K \in [0,1]$ are fixed and let
\[
c_1:= \frac{1}{\sqrt{12}} \left(\int_{0}^{1} \ex[\mathcal{L} f^{\diamond}(s, \tau_2,\ldots, \tau_K,W_s,W_{\tau_2}, \ldots, W_{\tau_K})^2] ds\right)^{1/2}. 
\]
Then it is
\begin{equation*}
\lim\limits_{n \rightarrow \infty} n \cdot \ex\left[\left(\int_{0}^{1} u_s dW_s-\ex[\int_{0}^{1} u_s dW_s|\P_n \cup \{W_{\tau_2},\ldots, W_{\tau_K}\}]\right)^2\right]^{1/2} = c_1.
\end{equation*}
\end{theorem}
The convergence rate in  Theorem \ref{theorem:OptimalApproxSkorohod} is essentially based on the exact approximation of the nonadapted parts $W_{\tau_2}, \ldots, W_{\tau_K}$ in $\ex[(I-\ex[I|\P_n \cup \{W_{\tau_2},\ldots, W_{\tau_K}\}]$.
Let us consider the following example of a stochastic exponential random variable without a finite nonadapted input:

\begin{example}
For a fixed $T \in (0,1) \setminus \Q$ we denote the fractional part as
\[
\{kT\} := kT - \lfloor kT \rfloor 
\]
and notice $\{kt\} \in (0,1) \setminus \Q$ for all $k \in \N$. Then we define the infinite chaos random variable depending on an infinite number of different nonadapted parts $W_{\{kT\}}$, $k \in \N$, as
\[
X_T = \sum_{k \geq 0} \frac{1}{k!} h^{k}_{\{kT\}}(W_{\{kT\}}) = \sum_{k \geq 0} \frac{1}{k!} W_{\{kT\}}^{\diamond k}. 
\]
Due to \eqref{eq:NormOFWickProductGaussian}, we observe
\[
\sum_{k=0}^{\infty} (k+1) \|\pi_k(X_T)\|^2_{L^2(\Omega)} = \sum_{k \geq 0} \frac{k+1}{k!} \{kT\}^k < \infty,
\]
which gives the existence of the Skorohod integral $\int_{0}^{1}u_s dW_s$ for the stationary integrand $(u_t= X_T)_{t \in [0,1]}$. One can easily check by \eqref{eq:L2NormWickProductInequality} that $X_T \in L^p(\Omega)$ for all $p>0$. Notice that $X_T$ is in some sense an infinite order Wick-analytic functional and by \cite[Proposition 10]{NP} the analytic representation 
\[
X_T = f(W_{0}, W_{\{T\}}, W_{\{2T\}}, \ldots) 
\]
is infinitely differentiable in all variables. 
Thus, in contrast to Theorem \ref{theorem:OptimalApproxSkorohod}, no finite set 

$\{W_{\tau_2},\ldots, W_{\tau_K}\}$ is sufficient for the proof technique in Theorem \ref{theorem:OptimalApproxSkorohod}.
\end{example}

This motivates to consider optimal approximation of Skorohod integrals with unattainable nonadapted parts. Surprisingly we have the following low convergence rate result in contrast to Theorem \ref{theorem:OptimalApproxSkorohod}. We consider the optimal approximation error with respect to the equidistant discretization
$$
e_{n} := \ex\left[\left(\int_{0}^{1} u_s dW_s - \ex\left[\left.\int_{0}^{1} u_s dW_s\right|\mathcal{P}_n\right]\right)^2\right]^{1/2}.
$$
The proofs of the following theorems are postponed to the next section.

\begin{theorem}\label{thm:OptApproxIrrational}
Suppose $K \in \N$, $f^{\diamond} \in C^{1,\ldots,1; 3,\ldots,3}([0,1]^K\times \R^{K})$,  $\tau_2, \ldots, \tau_K \in [0,1]$ are fixed and let
\begin{align*}
c_2 &= \frac{1}{2}  \left(\sum_{\tau_i \notin \Q}\ex\left[\left(\int_{0}^{1}\dfrac{\partial}{\partial x_i} f^{\diamond}(s,\tau_2, \ldots, \tau_K,W_{s}, W_{\tau_2}, \ldots, W_{\tau_K}) dW_s\right)^2\right]\right)^{1/2}.
\end{align*}
Then there exists a strictly increasing sequence of integers $(n_k)_{k \in \N}$, such that
\[
\lim_{k \rightarrow \infty} n_k^{1/2} e_{n_k} =c_2.
\]
\end{theorem}

\begin{remark}
The smoothness assumptions and Theorem \ref{theorem:SkorohodItoFormula} ensure the existence of all Skorohod integrals involved.
\end{remark}

A similar optimal rate appears for It\^o integrals of discontinuous integrands:

\begin{theorem}\label{theorem:MainTheoremIto}
Suppose $N \subset (0,1) \setminus \Q$ a countable set of discontinuities, $f \in C^{1,2}([0,1]\setminus N \times \R)$ such that $f(t,\cdot)$ is nonlinear for all fixed $t \in [0,1]\setminus N$. Then for the It\^o integral 
\[
I= \int_{0}^{1} f(s,W_s) dW_s,
\]
and the constant
\[
c_2 = \frac{1}{2}  \left(\int_{0}^{1}\ex\left[\left(\dfrac{\partial}{\partial x} f(s,W_{s})\right)^2\right]ds\right)^{1/2}, 
\]
there exists a strictly increasing sequence of integers $(n_k)_{k \in \N}$, such that
\[
\lim_{k \rightarrow \infty} n_k^{1/2} e_{n_k} = c_2.  
\]
\end{theorem}

\begin{example}
In particular, we obtain strictly increasing sequences $(n_k)_{k \in \N}$ such that, as $k$ tends to infinity:
\begin{enumerate}
 \item For the It\^o integral $I=\int_{0}^{1/\pi} e^{W_s}dW_s$,
 \begin{align*}
n_k \, e_{n_k}^2 \sim \frac{1}{8}(1-e^{-2/\pi}),
\end{align*}
\item For the Skorohod integral $I=\int_{0}^{1} e^{W_{1/\pi}}dW_s$,
\begin{align*}
n_k \, e_{n_k}^2 \sim \frac{1}{4}\left(1+\pi^{-2}\right)e^{2/\pi}.
\end{align*}
\end{enumerate}
\end{example}

\section{Proofs of Theorem \ref{thm:OptApproxIrrational} and \ref{theorem:MainTheoremIto}}\label{section_3}

In the proofs we will frequently make use of:
\begin{remark}\label{remark:NormComputations} 
The Brownian bridge is denoted by
\[
B^n_t := W_t - W_t^{\lin}
\]
and yields the expansion
\begin{equation}\label{eq:WickPowersBridge}
W_t^{\diamond k} - (W_{t}^{\lin})^{\diamond k} = B^n_t \diamond \sum\limits_{j=1}^{k} W_t^{\diamond k-j} \diamond (W_{t}^{\lin})^{\diamond j-1}. 
\end{equation}
Some elementary computations via \eqref{eq:NormOFWickProductGaussian} give for all $s,t \in [0,1]$, $i,j \in \{1,\ldots, n\}$:
\begin{align}
\displaybreak[0]
&\ex[W_s W_t^{\lin}] = \ex[W_s^{\lin} W_t^{\lin}] = \left.\begin{cases} \frac{\lfloor nt\rfloor}{m} + n(t-\frac{\lfloor nt\rfloor}{m})(s-\frac{\lfloor nt\rfloor}{m}) &,\lfloor ns\rfloor = \lfloor nt\rfloor\\ s\wedge t &,\lfloor ns\rfloor \neq \lfloor nt\rfloor
\end{cases}\right\} \leq s\wedge t,\label{eq_covariances1}\\
&\ex[B^n_s B^n_t]  = \ex[B^n_s W_t] = \left.\begin{cases} s\wedge t - \frac{\lfloor nt\rfloor}{m} -(s-\frac{\lfloor nt\rfloor}{m})(t-\frac{\lfloor nt\rfloor}{m})n &,\lfloor ns\rfloor = \lfloor nt\rfloor\\ 0 &,\lfloor ns\rfloor \neq \lfloor nt\rfloor
\end{cases}\right\} \leq \frac{1}{4n},\label{eq_covariances2}\\
&\ex[B^n_s W_t^{\lin}] =0,\label{eq_covariances3}\\
&\int_{(i-1)/n}^{i/n}\int_{(j-1)/n}^{i/n} \ex[B^n_s B^n_t] ds dt = \eins_{\{i=j\}} \frac{1}{12}n^{-3}.\label{eq_covariances4}
\end{align}
\end{remark}

The asymptotic bahaviour relies on the following worst case convergence rate of $\ex[(B^{n}_t)^2]$ as $n$ tends to infinity:

\begin{proposition}\label{prop:Weyl}
Suppose $t \in (0,1) \setminus \Q$. Then there exists a strictly increasing sequence of integers $(n_k)_{k \in \N}$ such that
\[
\ex[(B^{n_k}_t)^2] = \ex[B^{n_k}_t W_t] \sim \frac{1}{4n_k}.
\]
\end{proposition}

\begin{proof}
Due to the Weyl equidistribution theorem \cite{Weyl}, for every irrational $t \in (0,1)$ there exists a strictly increasing sequence of integers $(n_k)_{k \in \N}$ with 
$$
\lim_{k \rightarrow \infty} \left(n_k t - \lfloor n_k t\rfloor\right) = 1/2.
$$
Hence, by \eqref{eq_covariances2},
$$
\ex[(B^{n_k}_t)^2] = \ex[B^{n_k}_t W_t] = (t-\lfloor n_k t \rfloor/n_k)(1-n_k t - \lfloor n_k t \rfloor) \sim \frac{1}{4 n_k}.
$$
\end{proof}

We firstly consider the simpler Theorem \ref{theorem:MainTheoremIto}.

\begin{proof}[Proof of Theorem \ref{theorem:MainTheoremIto}]
We consider only the integrand $f(s,x) \eins_{[0,t]}(s)$ for some $f \in C^{1,2}([0,1] \times \R)$ and a fixed $t \in (0,1) \setminus \Q$. Then the statement of the theorem follows by linearity and the It\^o isometry. The Wick-analytic representation \eqref{eq:WienerChaosByWickProducts} gives
\[
f(s,W_s) = \sum_{k \geq 0} a_k(s) W_s^{\diamond k} 
\]
with $a_k(\cdot) \in C^1([0,1])$. Due to the Wiener chaos decomposition it suffices to consider the individual chaoses $a_k(s) W_s^{\diamond k}$. Thanks to \eqref{eq:NormOFWickProductGaussian} and Remark \ref{remark:NormComputations}, for all $s,t \in [0,1]$, we have
\begin{align}\label{eq:MainTheoremIto1}
\ex\left[\left(W_t^{\diamond k} - (W_{t}^{\lin})^{\diamond k}\right)\left(W_s^{\diamond k} - (W_{s}^{\lin})^{\diamond k}\right)\right] &-\ex[B^n_t B^n_s]\ex\left[\left(kW_t^{\diamond k-1}\right)\left(kW_s^{\diamond k-1}\right)\right] \rightarrow 0
\end{align}
as $n$ tends to infinity. Thus, by the integration by parts formula in \eqref{eq:IBPWick}, Remark \ref{remark:NormComputations} and \eqref{eq:MainTheoremIto1} along the subsequence in Proposition \ref{prop:Weyl},  we conclude 
\begin{align*}
\displaybreak[0]
&\ex\left[\left(\int_{0}^{t} a_k(s) W_s^{\diamond k} dW_s -\ex[\int_{0}^{t} a_k(s) W_s^{\diamond k}dW_s |\P_{n}]\right)^2\right]\\
&= \ex\left[\left(a_k(t) \frac{1}{k+1}\left(W_t^{\diamond k+1} - (W_{t}^{\lin})^{\diamond k+1}\right)  -\sum_{i=1}^{n} \int_{(i-1)/n \wedge t}^{i/n \wedge t}a_k'(s) \left(W_s^{\diamond k+1} - (W_{s}^{\lin})^{\diamond k+1}\right)ds\right)^2\right]\\
&\sim \ex[(B^n_t)^2]\ex\left[\left(a_k(t) W_t^{\diamond k}  -\int_{0}^{t}a_k'(s) W_s^{\diamond k} ds\right)^2\right] \sim \frac{1}{4 n} \ex\left[\left(\int_{0}^{t} a_k(s) k W_s^{\diamond k-1} dW_s\right)^2\right].
\end{align*}
Hence, $a_k(s) k W_s^{\diamond k-1} = \dfrac{\partial}{\partial x}a_k(s) W_s^{\diamond k}$ and the It\^o isometry yield the asserted constant.
\end{proof}

\begin{proof}[Proof of Theorem \ref{thm:OptApproxIrrational}]
The proof is divided into three steps. In the first step we consider the simplest case of nonadapted integrands
$$
F^{\diamond}(W_s, W_{\tau_2})
$$
for some Wick-analytic functional $F^{\diamond}$. Then, in the second step, we conclude the statement for Wick-analytic functionals
$$
F^{\diamond}(W_s, W_{\tau_2}, \ldots, W_{\tau_K}).
$$
Finally, in the last step, we complete the proof to arbitrary integrands in the theorem.

\textit{Step 1}: \ Let the simple integrand
$$
F^{\diamond}(W_s,W_{\tau_2}) = \sum\limits_{k\geq 0, l>0} a_{k,l} l \, W_s^{\diamond (l-1)} \diamond W_{\tau_2}^{\diamond k}.
$$
(This strange form is chosen to simplify the integration by parts formula in \eqref{eq:IBPWick})
Firstly, due to the expansion in \eqref{eq:WickPowersBridge}, \eqref{eq:NormOFWickProductGaussian} and \eqref{eq_covariances1}-\eqref{eq_covariances3}, for $k,k', l, l' \in \N$, $t\in [0,1]$, we observe
\begin{align}
\displaybreak[0]
&\ex\left[\left((W_t^{\diamond k} -(W_{t}^{\lin})^{\diamond k})\diamond W_1^{\diamond l}\right)\left((W_t^{\diamond k'} -(W_{t}^{\lin})^{\diamond k'})\diamond W_1^{\diamond l'}\right)\right]\nonumber\\
&= \sum\limits_{\substack{i=1,\ldots, k\\ j=1,\ldots, k'}} \ex\left[\left(B^n_t \diamond  W_t^{\diamond k-i} \diamond (W_{t}^{\lin})^{\diamond i-1} \diamond W_1^{\diamond l}\right) \left(B^n_t \diamond  W_t^{\diamond k'-j} \diamond (W_{t}^{\lin})^{\diamond j-1} \diamond W_1^{\diamond l'}\right)\right]\nonumber\\
&= \ex[(B^n_t)^2]  \sum\limits_{\substack{i=1,\ldots, k\\ j=1,\ldots, k'}} \ex\left[\left(W_t^{\diamond k-i} \diamond (W_{t}^{\lin})^{\diamond i-1} \diamond W_1^{\diamond l}\right) \left(W_t^{\diamond k'-j} \diamond (W_{t}^{\lin})^{\diamond j-1} \diamond W_1^{\diamond l'}\right)\right]\nonumber\\
&\quad + \ex[B^n_t W_t]^2 \sum\limits_{\substack{i=1,\ldots, k-1\\ j=1,\ldots, k'-1}} (k-i)(k'-j)\nonumber\\ 
&\hspace*{3 cm} \times \ex\left[\left(W_t^{\diamond k-i-1} \diamond (W_{t}^{\lin})^{\diamond i-1} \diamond W_1^{\diamond l}\right) \left(W_t^{\diamond k'-j-1} \diamond (W_{t}^{\lin})^{\diamond j-1} \diamond W_1^{\diamond l'}\right)\right].\label{eq:MainTheorem1}
\end{align}
Making use of \eqref{eq:NormOFWickProductGaussian} and \eqref{eq:MainTheorem1}, we have the two sums
\begin{align}
\displaybreak[0]
e_n^2&=\ex\left[\left(\sum\limits_{k, l>0} a_{k,l} (W_{\tau_2}^{\diamond k} - (W_{\tau_2}^{\lin})^{\diamond k})\diamond W_1^{\diamond l}\right)^2\right]\nonumber\\
\displaybreak[0]
&= \ex[(B^n_{\tau_2})^2]\sum\limits_{k+l=k'+l' \geq 1} a_{k,l} a_{k',l'}\nonumber\\
&\qquad \times \ex\left[\left(\sum\limits_{i=1}^{k} W_{\tau_2}^{\diamond k-i} \diamond (W_{\tau_2}^{\lin})^{\diamond i-1} \diamond W_1^{\diamond l}\right) \left(\sum\limits_{i'=1}^{k'} W_{\tau_2}^{\diamond k'-i'} \diamond (W_{\tau_2}^{\lin})^{\diamond i'-1} \diamond W_1^{\diamond l'}\right)\right]\nonumber\\
&+ \ex[B^n_{\tau_2} W_{\tau_2}]^2 \sum\limits_{k+l=k'+l' \geq 1} a_{k,l} a_{k',l'} \nonumber\\
&\times \ex\left[\left(\sum\limits_{i=1}^{k} (k-i)W_{\tau_2}^{\diamond k-i-1} \diamond (W_{\tau_2}^{\lin})^{\diamond i-1} \diamond W_1^{\diamond l}\right) \left(\sum\limits_{i'=1}^{k'} (k'-i')W_{\tau_2}^{\diamond k'-i'-1} \diamond (W_{\tau_2}^{\lin})^{\diamond i'-1} \diamond W_1^{\diamond l'}\right)\right]\nonumber\\ 
&=:X^n_1 + X^n_2.\label{eq:MainTheorem2}
\end{align}
Obviously, if $\tau_2 \in \Q$, then we can choose a strictly increasing sequence of indices $n \in \N$ such that $B^n_{\tau_2} =0$ and $e_n^2=0$.

A computation via \eqref{eq:NormOFWickProductGaussian} and \eqref{eq:WickPowersBridge} shows that the $L^2$-norm of the difference
$$
\left(\sum\limits_{i=1}^{k} W_{\tau_2}^{\diamond k-i} \diamond (W_{\tau_2}^{\lin})^{\diamond i-1} \diamond W_1^{\diamond l}\right) - \left(k W_{\tau_2}^{\diamond (k-1)} \diamond W_1^{\diamond l}\right)
$$
behaves in upper bounds as the $L^2$-norm of
$$
B^n_{\tau_2} \diamond k(k-1) W_{\tau_2}^{\diamond (k-1)} \diamond W_1^{\diamond l}.
$$
Hence,
via \eqref{eq_covariances2}-\eqref{eq_covariances4} in Remark \ref{remark:NormComputations} and the assumptions on the integrand, we conclude for
\begin{align*}
f_n^2 &:= \ex[(B^n_{\tau_2})^2]\sum\limits_{k+l=k'+l' \geq 1} a_{k,l} a_{k',l'}\ex\left[\left(k W_{\tau_2}^{\diamond (k-1)} \diamond W_1^{\diamond l}\right) \left(k' W_{\tau_2}^{\diamond (k'-1)} \diamond W_1^{\diamond l'}\right)\right]\\
&=\ex[(B^n_t)^2]  \ex\left[\left(\int_{0}^{1} \dfrac{\partial}{\partial x_2} F^{\diamond}(W_s,W_{\tau_2}) dW_s\right)^2\right],
\end{align*}
that
\begin{align}\label{eq:MainTheorem3}
|X^n_1 - f_n^2| \in \Oo(n^{-2}). 
\end{align}
By a similar reasoning via \eqref{eq:NormOFWickProductGaussian} and Remark \ref{remark:NormComputations} for 
\begin{align*}
g_n^2 &:= \ex[B^n_{\tau_2} W_{\tau_2}]^2 \sum\limits_{k+l=k'+l' \geq 1} a_{k,l} a_{k',l'} \ex\left[\left(\binom{k}{2} W_{\tau_2}^{\diamond (k-2)} \diamond W_1^{\diamond l}\right) \left(\binom{k'}{2} W_{\tau_2}^{\diamond (k'-2)}\diamond W_1^{\diamond l'}\right)\right]\\
&= \ex[B^n_{\tau_2} W_{\tau_2}]^2  \ex\left[\left(\int_{0}^{1} \dfrac{\partial^2}{\partial x_2^2} F^{\diamond}(W_s,W_{\tau_2}) dW_s\right)^2\right]
\end{align*}
we have
\begin{align}\label{eq:MainTheorem4}
|X^n_2 - g_n^2| \in \Oo(n^{-3}).
\end{align}
The Skorohod integrals $\int_{0}^{1} \dfrac{\partial}{\partial x_2} F^{\diamond}(W_s,W_{\tau_2}) dW_s$ and $\int_{0}^{1} \dfrac{\partial^2}{\partial x_2^2} F^{\diamond}(W_s,W_{\tau_2}) dW_s$ involved exist by the continuity assumptions and Theorem \ref{theorem:SkorohodItoFormula}.
Hence, by \eqref{eq:MainTheorem2}-\eqref{eq:MainTheorem4}, we conclude
$$
e_n^2 = f_n^2 + \Oo(n^{-2}).
$$
With the sequence in Proposition \ref{prop:Weyl}, this yields
\begin{align}
\displaybreak[0]
\lim_{k \rightarrow \infty} n_k \cdot e_{n_k}^2= \frac{1}{4}\ex\left[\left(\int_{0}^{1} \dfrac{\partial}{\partial x_2} F^{\diamond}(W_s,W_{t}) dW_s\right)^2\right].\label{eq:MainTheorem5}
\end{align}

\textit{Step 2}: \ Now we make use of the shorthand notations 
\begin{align}
\bar{l} := (l_1,\ldots, l_K) \in \N^K,\quad |\bar{l}| :=\sum l_i,\quad  a'_{\bar{l}}(s) := \dfrac{\partial}{\partial s}a_{l_1,\ldots, l_K}(s).\label{eq:MainTheoremShortance} 
\end{align}
Dealing with a Wick-analytic functional integrand
$$
F^{\diamond}(W_s, W_{\tau_2}, \ldots, W_{\tau_K}) = \sum\limits_{l_1, \ldots, l_K \geq 0} a_{\bar{l}} \, l_1 W_s^{\diamond (l_1-1)} \diamond W_{\tau_2}^{\diamond l_2} \diamond \cdots \diamond W_{\tau_K}^{\diamond l_K},
$$
we make use of the multivariate expansion of \eqref{eq:MainTheorem2} as
\begin{align}
\displaybreak[0]
&W_{\tau_2}^{\diamond l_2} \diamond \cdots \diamond W_{\tau_K}^{\diamond l_K} - (W_{\tau_2}^\lin)^{\diamond l_1} \diamond \cdots \diamond (W_{\tau_K}^{\lin})^{\diamond l_K}\nonumber\\ 
&= \sum\limits_{I=2}^{K} B^n_{\tau_I} \diamond \sum\limits_{i=1}^{l_I} W_{\tau_I}^{\diamond l_I-i} \diamond (W_{\tau_{I}}^{\lin})^{\diamond i-1} \diamond W_{\tau_2}^{\diamond l_2} \diamond \cdots \diamond W_{\tau_{I-1}}^{\diamond l_{I-1}} \diamond  (W_{\tau_{I+1}}^{\lin})^{\diamond l_{I+1}} \diamond \cdots \diamond  (W_{\tau_{K}}^{\lin})^{\diamond l_{K}}.\label{eq:MainTheorem6}
\end{align}
Then, via
\eqref{eq:NormOFWickProductGaussian}, Remark \ref{remark:NormComputations} and \eqref{eq:MainTheorem6}, for appropriate and sufficiently large $n$ we obtain similarly to \eqref{eq:MainTheorem2},
\begin{align*}
\displaybreak[0]
e_n^2&=\sum\limits_{\tau_I\notin \Q} \ex[(B^n_{\tau_I})^2]  \sum\limits_{|\bar{l}|=|\bar{l}'| \geq 1} a_{\bar{l}} a_{\bar{l}'}\nonumber\\
&\quad \times \ex\left[\left(\sum_{i=1}^{l_{I}}  W_{\tau_I}^{\diamond l_I-i} \diamond (W_{\tau_I}^{\lin})^{\diamond i-1} \diamond W_{\tau_2}^{\diamond l_1} \diamond \cdots \diamond W_{\tau_{I-1}}^{\diamond l_{I-1}} \diamond  
\cdots \diamond  (W_{\tau_{K}}^{\lin})^{\diamond l_{K}} \diamond W_1^{\diamond l_1}\right)\right.\nonumber\\
&\hspace*{1.5 cm} \times \left. \left(\sum_{i'=1}^{l_{I}'} W_{\tau_{I}}^{\diamond l'_{I}-i'} \diamond (W_{\tau_{I}}^{\lin})^{\diamond i'-1} \diamond W_{\tau_2}^{\diamond l'_2} \diamond \cdots \diamond W_{\tau_{I-1}}^{\diamond l'_{I'-1}} \diamond  
\cdots \diamond  (W_{\tau_{K}}^{\lin})^{\diamond l'_{K}} \diamond W_1^{\diamond l'_1}\right)\right]\nonumber\\
&+ \sum\limits_{\tau_I \notin \Q} \ex[(B^n_{\tau_I}W_{\tau_{I}})^2] \sum\limits_{|\bar{l}|=|\bar{l}'| \geq 1} a_{\bar{l}} a_{\bar{l}'}\nonumber\\
&\quad \times \ex\left[\left(\sum_{i=1}^{l_{I}} (l_I-i) W_{\tau_I}^{\diamond l_I-i-1} \diamond (W_{\tau_I}^{\lin})^{\diamond i-1} \diamond W_{\tau_2}^{\diamond l_1} \diamond \cdots \diamond W_{\tau_{I-1}}^{\diamond l_{I-1}} \diamond  
\cdots \diamond W_1^{\diamond l_1}\right)\right.\nonumber\\
& \qquad \quad \times \left. \left(\sum_{i'=1}^{l_{I}'} (l_{I}'-i')W_{\tau_{I}}^{\diamond l'_{I}-i'-1} \diamond (W_{\tau_{I}}^{\lin})^{\diamond i'-1} \diamond W_{\tau_2}^{\diamond l'_2} \diamond \cdots \diamond W_{\tau_{I-1}}^{\diamond l'_{I-1}} \diamond  
\cdots \diamond W_1^{\diamond l'_1}\right)\right]\nonumber\\
&=:X^n_1 + X^n_2.
\end{align*}
Here the indices $I$ with $\tau_I \in \Q$ are ignored due to $B^n_{\tau_I}=0$ for an appropriate sequence of indices. Moreover, for sufficiently large $n$ we have  $\lfloor n\tau_{I} \rfloor \neq \lfloor n\tau_{I'}\rfloor$ for all $\tau_{I} \neq \tau_{I'}$ and thus, via \eqref{eq_covariances2}, 
$$
\ex[B^n_{\tau_{I'}}W_{\tau_{I}}] = \ex[B^n_{\tau_{I'}}B^n_{\tau_{I}}]=0.
$$
Thus, by an analogous reasoning as for \eqref{eq:MainTheorem3} and the assumptions on the integrand, we conclude
$$
|X^n_1 - f_n^2| \in \Oo(n^{-2})
$$
for
\begin{align*}
f_n^2 &:= \sum\limits_{\tau_I\notin \Q} \ex[(B^n_{\tau_I})^2]  \sum\limits_{|\bar{l}|=|\bar{l}'| \geq 1} a_{\bar{l}} a_{\bar{l}'}\ex\left[\left(W_{\tau_2}^{\diamond l_1} \diamond \cdots \diamond  l_I W_{\tau_I}^{\diamond l_I-1} \diamond
\cdots \diamond  (W_{\tau_{K}}^{\lin})^{\diamond l_{K}} \diamond W_1^{\diamond l_1}\right)\right.\nonumber\\
& \qquad \qquad \qquad \qquad \qquad \qquad \qquad \times \left. \left(W_{\tau_2}^{\diamond l_1} \diamond \cdots \diamond   l_I' W_{\tau_I}^{\diamond l_I'-1} \diamond
\cdots \diamond  (W_{\tau_{K}}^{\lin})^{\diamond l_{K}} \diamond W_1^{\diamond l_1}\right)\right]\nonumber\\
&=\sum\limits_{\tau_I \notin \Q} \ex[(B^n_{\tau_I})^2]\ex\left[\left(\int_{0}^{1}\dfrac{\partial}{\partial x_I} F^{\diamond}(W_s, W_{\tau_2}, \ldots, W_{\tau_K}) dW_s\right)^2\right].
\end{align*}
Similarly to \eqref{eq:MainTheorem4}, we have
$$
|X_2^n - g_n^2| \in \Oo(n^{-3})
$$
for 
\begin{align*}
g_n^2 &:= \sum\limits_{\tau_I \notin \Q} \ex[B^n_{\tau_2} W_{\tau_2}]^2 \ex\left[\left(\int_{0}^{1}\dfrac{\partial^2}{\partial x_I^2} F^{\diamond}(W_s, W_{\tau_2}, \ldots, W_{\tau_K}) dW_s\right)^2\right].
\end{align*}
Proposition \ref{prop:Weyl} yields again an appropriate sequence with
\begin{align*}
\lim_{k \rightarrow \infty} n_k \cdot e_{n_k}^2= \frac{1}{4} \sum_{\tau_i \notin \Q} \ex\left[\left(\int_{0}^{1}\dfrac{\partial}{\partial x_i} F^{\diamond}(W_s, W_{\tau_2}, \ldots, W_{\tau_K}) dW_s\right)^2\right].
\end{align*}
We notice that all Skorohod integrals above exist by the Skorohod It\^o formula in Theorem \ref{theorem:SkorohodItoFormula}.\\

\textit{Step 3}: \ The generalization to arbitrary integrands follows by a straightforward extension of the computations above and the Wiener chaos decomposition. For the  arbitrary integrand, we have
\begin{equation*}
f(t, \tau_2\ldots, \tau_K,W_{t}, \ldots, W_{\tau_K}) = \sum_{l_1, \ldots, l_K\geq 0} a_{l_1, \ldots, l_K}(t) W_{t}^{\diamond l_1} \diamond W_{\tau_2}^{\diamond l_2} \diamond \cdots \diamond W_{\tau_K}^{\diamond l_K},
\end{equation*}
where the continuity assumption implies for all coefficients $a_{l_1,\ldots, l_K}(\cdot) \in C^1([0,1])$. We consider exemplary the integrand 
\[
f(s,t,W_s, W_{t}) = a(s)\, l\, W_s^{\diamond (l-1)} \diamond W_t^{\diamond m}                                                                                                                                                                                                                                                                                                                                                                                                                      \]
for some fixed $t \in (0,1)\setminus \Q$, $m,l \geq 1$. The generalization is straightforward. Thanks to the integration by parts formula \eqref{eq:IBPWick}, we have
\begin{align*}
\displaybreak[0]
e_n^2 &=\ex\left[\left(a(1)W_1^{\diamond l} \diamond ((W_t^{\diamond m} - (W_{t}^{\lin})^{\diamond m})\right.\right.\\
&\left.\left.\quad - \sum_{i=1}^{n} \int_{(i-1)/n}^{i/n} a'(s) (W_s^{\diamond l} \diamond W_t^{\diamond m} - (W_s^{\lin})^{\diamond l} \diamond (W_t^{\lin})^{\diamond m})ds\right)^2\right]\\
&= \ex\left[\left(a(1)W_1^{\diamond l} \diamond ((W_t^{\diamond m} - (W_{t}^{\lin})^{\diamond m})\right)^2\right]\\
&-2\sum_{i=1}^{n} \int_{\frac{i-1}{n}}^{\frac{i}{n}} a'(s)a(1)\ex\left[\left(W_1^{\diamond l} \diamond ((W_t^{\diamond m} - (W_{t}^{\lin})^{\diamond m})\right)\left(W_s^{\diamond l} \diamond W_t^{\diamond m} - (W_s^{\lin})^{\diamond l} \diamond (W_t^{\lin})^{\diamond m}\right)\right]ds\\
&+ \ex\left[\left(\sum_{i=1}^{n} \int_{(i-1)/n}^{i/n} a'(s) (W_s^{\diamond l} \diamond W_t^{\diamond m} - (W_s^{\lin})^{\diamond l} \diamond (W_t^{\lin})^{\diamond m})ds\right)^2\right]\\
&=: I^n_1 -2I^n_2+I^n_3.
\end{align*}
By \eqref{eq:MainTheorem5} in Step 1, for the sequence of integers in Proposition \ref{prop:Weyl}, we have
\[
n_k \cdot I^{n_k}_1 - \frac{1}{4} \ex\left[\left(a(1) W_1^{\diamond l} \diamond mW_t^{\diamond m-1}\right)^2\right] \rightarrow 0,
\]
as $k$ tends to infinity. By the expansion \eqref{eq:MainTheorem6} in Step 2 and \eqref{eq:MainTheorem5}, we conclude similarly
\begin{align*}
&n_k \cdot I^{n_k}_2 - \frac{1}{4} \ex\left[\left(a(1) W_1^{\diamond l} \diamond mW_t^{\diamond m-1}\right) \left(\int_{0}^{1}a'(s) W_s^{\diamond l} \diamond m W_t^{\diamond m-1} ds\right)\right] \rightarrow 0,\\
&n_k \cdot I^{n_k}_3 - \frac{1}{4} \ex\left[ \left(\int_{0}^{1}a'(s) W_s^{\diamond l} \diamond m W_t^{\diamond m-1} ds\right)^2\right] \rightarrow 0.
\end{align*}
Due to Proposition \ref{prop:SkorohodAndWick} and the integration by parts formula in \eqref{eq:IBPWick},
$$
a(1) W_1^{\diamond l} \diamond mW_t^{\diamond m-1} - \int_{0}^{1}a'(s) W_s^{\diamond l} \diamond m W_t^{\diamond m-1} ds = \int_{0}^{1} a(s) l W_s^{\diamond l} \diamond m W_t^{\diamond m-1} dW_s.
$$
Hence we conclude
\[
\lim_{k \rightarrow \infty} n_k e_{n_k}^2 = \frac{1}{4}\ex\left[\left(\int_{0}^{1} a(s) l W_s^{\diamond l} \diamond m W_t^{\diamond m-1} dW_s\right)^2\right]= \frac{1}{4}\ex\left[\left(\int_{0}^{1}\dfrac{\partial}{\partial x_2} f(s,t,W_s, W_{t}) dW_s\right)^2\right].
\]
Thanks to the same arguments as above and in Step 1 and Step 2, the Wiener chaos decomposition completes the proof of the statement for arbitrary integrands $f \in C^{1,\ldots,1;3,\ldots, 3}([0,1]^K\times \R^{K})$. The existence of all Skorohod integrals involved is justified by Theorem \ref{theorem:SkorohodItoFormula}. 
\end{proof}

\section{Examples of further irregularity}\label{section_4}

In this section we discuss optimal approximation for Skorohod integrals of integrands beyond the continuity conditions in Theorem \ref{theorem:OptimalApproxSkorohod} and Theorem \ref{thm:OptApproxIrrational}. In fact, we present Skorohod integrals $I$
such that the mean square error
$$
e_n = \ex[(I - \ex[I|\mathcal{P}_n])^2]^{1/2}
$$
exhibits a lower convergence rate than $n^{-1/2}$. Moreover we construct Skorohod integrals with infinite nonadapted part (i.e. an integrand which depends on an infinite set $\{W_{t}, t>0\}$) and optimal rates $n^{-\alpha}$ with $\alpha \in (0,1)$.



\begin{proposition}\label{prop:OptimalApproxAbsoluteValue}
Let $t \in [0,1]$ and the Skorohod integral
\begin{align}\label{eq:SkorohodAbsValue} 
I=\int_{0}^{1} W_t| dW_s = |W_t| \diamond W_1.
\end{align}
Then there exists in each case a strictly increasing sequence of integers $(n_k)_{k \in \N}$ such that:
\begin{enumerate}
 \item[(i)] If $t \in \Q$, then $e_{n_k}=0$.
 \item[(ii)] If $t \notin \Q$, then for all $\ve\in (0,1/4)$,
\begin{equation}\label{eq:thmOptApproxAbsValue1}
\frac{t^{3/4}}{\sqrt{2\pi}} \, n_k^{-1/4} \leq e_{n_k} \in \Oo(n_k^{-1/4+\ve}).
\end{equation}
\end{enumerate}
\end{proposition}

\begin{proof}
 $(i)$ \ The Wiener chaos decomposition (see e.g. \cite[p. 65]{Kuo}) gives
\begin{equation}\label{eq:absWtWienerChaos}
|W_t| =  \sqrt{\frac{2t}{\pi}} \sum\limits_{m \geq 0} \frac{(-1)^{m+1} t^{-m}}{(2m-1) (2m)!!}\, W_t^{\diamond 2m}.
\end{equation}
Applying Proposition \ref{prop:WickConditionalExpectation} on \eqref{eq:absWtWienerChaos}, we have
\begin{equation}\label{eq:absWtWienerChaosApprox}
\ex[I|\mathcal{P}_n] = \sqrt{\frac{2t}{\pi}} \sum\limits_{m \geq 0} \frac{(-1)^{m+1} t^{-m}}{(2m-1) (2m)!!}\, (W_t^{\lin})^{\diamond 2m} \diamond W_1.
\end{equation}
Thus, for $t \in \frac{1}{m}\N$ and $n \in m\N$ we obtain $W^{\lin}_t=W_t$ and exact simulation. 

$(ii)$ \ We firstly observe by the orthogonality of $(W_t^{\diamond 2m} -(W_{t}^{\lin})^{\diamond 2m})\diamond W_1$ for different $m$,
\begin{align}
e_n^2 &= \frac{2t}{\pi} \sum\limits_{m \geq 1} \frac{ t^{-2m}}{(2m-1)^2 (2m)!!^2}\ex\left[\left((W_t^{\diamond 2m} -(W_{t}^{\lin})^{\diamond 2m})\diamond W_1\right)^2\right].
\label{eq:OptApproxAbsValue1}
\end{align}
Due to \eqref{eq:MainTheorem1}, the proof is based on upper and lower bounds of 
\[
\ex\left[\left(W_t^{\diamond k-i} \diamond (W_{t}^{\lin})^{\diamond i-1} \diamond W_1\right) \left(W_t^{\diamond k-j} \diamond (W_{t}^{\lin})^{\diamond j-1} \diamond W_1\right)\right]. 
\]

Lower bounds:\ We define for shorthand
\[
z:= \frac{t-\ex[(B^n_t)^2]}{t} \leq 1.
\]
Due to Remark \ref{remark:NormComputations}, it is
\begin{align}
&\min_{X,Y \in \{W_t,W_1,W_t^{\lin}\}}\ex[XY] = \ex[(W_t^{\lin})^2] = t-\ex[(B^n_t)^2] = tz.\label{eq:OptApproxAbsValue2}
\end{align}
Hence, by \eqref{eq:NormOFWickProductGaussian} and \eqref{eq:OptApproxAbsValue2}, we have
\begin{align*}
\displaybreak[0]
&\ex\left[\left(W_t^{\diamond k-i} \diamond (W_{t}^{\lin})^{\diamond i-1} \diamond W_1\right) \left(W_t^{\diamond k-j} \diamond (W_{t}^{\lin})^{\diamond j-1} \diamond W_1\right)\right]\nonumber\\
&\qquad \geq \ex\left[\left(W_t^{\diamond k+1-i} \diamond (W_{t}^{\lin})^{\diamond i-1}\right) \left(W_t^{\diamond k+1-j} \diamond (W_{t}^{\lin})^{\diamond j-1} \right)\right] \geq k! \, (tz)^k.
\end{align*}
Thanks to \eqref{eq:MainTheorem1}, ignoring the second term on the right hand side in \eqref{eq:MainTheorem1}, we obtain 
\begin{align}
&\ex\left[\left((W_t^{\diamond 2m} -(W_{t}^{\lin})^{\diamond 2m})\diamond W_1\right)^2\right] \geq \ex[(B^n_t)^2] (2m)^2 (2m)! \, (tz)^{2m}.\label{eq:OptApproxAbsValue3} 
\end{align}
We observe for sufficiently large $n$ that $z\geq 0.96$ and therefore
\[
\frac{1}{\sqrt{1-z^2}}- 1 = \frac{1}{\sqrt{1-z}}\left(\frac{z^2}{\sqrt{1+z}(1+\sqrt{1-z^2})}\right) \geq \frac{1}{2\sqrt{1-z}}.
\]
Then, making use of  \eqref{eq:OptApproxAbsValue1}, \eqref{eq:OptApproxAbsValue3}, the series
\begin{align}\label{eq:SeriesBinom}
\sum_{m\geq 0} \frac{(2m)! z^{2m}}{(2m)!!^2} = \sum_{m\geq 0} \binom{2m}{m} \left(\frac{z^2}{4}\right)^m = \frac{1}{\sqrt{1-z^2}}
\end{align}
(see e.g. \cite[2.1]{Aigner}) and $\ex[(B^n_t)] = t(1-z)$, we conclude for all $\ve \in (0,1/2]$ and sufficiently large n:
\begin{align}
\displaybreak[0]
e_n^2 
&= \frac{2t}{\pi} \sum\limits_{m \geq 1} \frac{ t^{-2m}}{(2m-1)^2 (2m)!!^2}\ex\left[\left((W_t^{\diamond 2m} -(W_{t}^{\lin})^{\diamond 2m})\diamond W_1\right)^2\right]\nonumber\\
&\geq \ex[(B^n_t)^2] \frac{2t}{\pi} \sum\limits_{m \geq 1} \binom{2m}{m}\left(\frac{z^2}{4}\right)^m = \ex[(B^n_t)^2] \frac{2t}{\pi} \left(\frac{1}{\sqrt{1-z^2}}- 1\right)\nonumber\\
&\geq
\ex[(B^n_t)^2]^{1-\ve} \left(\frac{t^{1+\ve}}{\pi} (1-z)^{\ve-1/2}\right).\label{eq:OptApproxAbsValue4}
\end{align}
For $\ve =1/2$ and the sequence from Proposition \ref{prop:Weyl} we conclude the inequality in \eqref{eq:thmOptApproxAbsValue1}. Otherwise the term $\frac{t^{1+\ve}}{4^{1-\ve}\pi} (1-z)^{\ve-1/2}$ is finite for $0<z<1$ and explodes as $z$ tends to $1$. Thus with the strictly increasing sequence of integers $(n_k)_{k \in \N}$ in Proposition \ref{prop:Weyl}, for all $\ve \in (0,1/2)$ there is no upper bound on the term in \eqref{eq:OptApproxAbsValue4} and we obtain $e_{n_k}^2 \notin\Oo(n_k^{-1/2-\ve})$.\\

Upper bounds:\ Via \eqref{eq:MainTheorem1}, we obtain
\begin{align}
\displaybreak[0]
&\ex\left[\left((W_t^{\diamond k} -(W_{t}^{\lin})^{\diamond k})\diamond W_1\right)^2\right]\nonumber\\ 
&\leq \ex[(B^n_t)^2] \sum\limits_{i,j=1}^{k} \ex\left[\left(W_t^{\diamond k-i} \diamond (W_{t}^{\lin})^{\diamond i-1} \diamond W_1\right) \left(W_t^{\diamond k-j} \diamond (W_{t}^{\lin})^{\diamond j-1} \diamond W_1\right)\right] + \Oo(n^{-2}).\label{eq:OptApproxAbsValueUp}
\end{align}
Thanks to \eqref{eq:NormOFWickProductGaussian} and \eqref{eq_covariances1}, we have
\begin{align}
&\ex\left[\left(W_t^{\diamond k-i} \diamond (W_{t}^{\lin})^{\diamond i-1} \diamond W_1\right) \left(W_t^{\diamond k-j} \diamond (W_{t}^{\lin})^{\diamond j-1} \diamond W_1\right)\right]\leq k! t^{k-1} z^{(i \vee j -1)}.\label{eq:OptApproxAbsValue5}
\end{align}
Since
\begin{align}\label{eq:OptApproxAbsValue6}
\sum_{i,j=1}^{k} z^{i\vee j -1} = 2\sum_{j=1}^{k} (j-1)z^{j-1} + \sum_{j=1}^{k}z^{j-1} 
\end{align}
and $\sum_{j=1}^{k}z^{j-1} \leq k$, it remains to control the first sum on the right hand side. For every $\ve \in (0,1/2)$, the H\"older inequality implies
\begin{align}
\displaybreak[0]
\sum_{j=1}^{k} (j-1)z^{j-1} &\leq \left(\sum_{j=1}^{k}z^{\frac{j-1}{1/2+\ve}}\right)^{1/2+\ve}\left(\sum_{j=1}^{k}(j-1)^{\frac{1}{1/2-\ve}}\right)^{1/2-\ve}\nonumber\\
&\leq \left(\frac{1-z^{\frac{k}{1/2+\ve}}}{1-z^{\frac{1}{1/2+\ve}}}\right)^{1/2+\ve}\left(\int_{0}^{k} x^{\frac{1}{1/2-\ve}}dx \right)^{1/2-\ve}.\label{eq:OptApproxAbsValue7}
\end{align}
For every $t \in (0,1)$ it is $z \in (0,1]$ for sufficiently large $n$. Thus, by $1/(1/2+\ve)>1$ we observe
\begin{align*}
 \displaybreak[0]
(1-z)^{1/2+\ve} \left(\frac{1-z^{\frac{k}{1/2+\ve}}}{1-z^{\frac{1}{1/2+\ve}}}\right)^{1/2+\ve} 
\leq 1.
\end{align*}
Hence, by \eqref{eq:OptApproxAbsValue6}-\eqref{eq:OptApproxAbsValue7}, $(\int_{0}^{k} x^{\frac{1}{1/2-\ve}}dx)^{1/2-\ve} < k^{3/2-\ve}$  and $\ex[(B^n_t)^2] = t(1-z)$, we obtain
\begin{align*}
\ex[(B^n_t)^2]  \sum_{i,j=1}^{k} z^{i\vee j -1} \leq \ex[(B^n_t)^2]^{1/2-\ve} t^{1/2+\ve}\, 3 \, k^{3/2-\ve}.
\end{align*}
Plugging this into \eqref{eq:OptApproxAbsValueUp}-\eqref{eq:OptApproxAbsValue5}, we conclude
\begin{align*}
\displaybreak[0]
&\ex\left[\left((W_t^{\diamond k} -(W_{t}^{\lin})^{\diamond k})\diamond W_1\right)^2\right] \leq \ex[(B^n_t)^2]^{1/2-\ve} \,3 \, k! \, k^{3/2-\ve}t^{k-1}  + \Oo(n^{-2}).
\end{align*}
Due to \eqref{eq:OptApproxAbsValue1} and the Stirling formula
\begin{align}\label{eq:StirlingFormula}
\frac{1}{4^m}\binom{2m}{m} \sim \frac{1}{\sqrt{\pi m}},
\end{align}
we obtain a constant $c>0$ such that 
\begin{align*}
e_n^2 &\leq \ex[(B^n_t)^2]^{1/2-\ve} c \sum\limits_{m \geq 1} m^{-(1+\ve)}  + \Oo(n^{-2}).
\end{align*}
By the sequence $(n_k)_{k \in \N}$ in Proposition \ref{prop:Weyl}, this gives $e_{n_k}^2 \in \Oo(n_k^{-1/2+\ve})$ for all $\ve \in (0,1/2)$.
\end{proof}

\begin{remark}
 Due to  \eqref{eq:NormOFWickProductGaussian} and \eqref{eq_covariances1},  for all $0 \leq i,j \leq k$, we observe
the hypergeometric series
\begin{align}
\displaybreak[0]
&\ex\left[\left(W_t^{\diamond k-i} \diamond (W_{t}^{\lin})^{\diamond i} \right) \left(W_t^{\diamond k-j} \diamond (W_{t}^{\lin})^{\diamond j} \right)\right]\nonumber\\
&= \sum_{l= 0 \vee \left(k-(i+j)\right)}^{k-(i \vee j)}  t^k z^{k-l}\frac{(k-i)! \, (k-j)! \, i!\, j!}{l!\, (k-i-l)!\, (k-j-l)!\, (i+j-k+l)!}.\label{eq:HypergeometricSeries} 
\end{align}
We used the simple bounds $\geq k! t^kz^k$ and $\leq k!t^k z^{i\vee j}$. 
From the proof and computer experiments we conjecture 
\[
e_{n_k} \sim c \, n_k^{-1/4} 
\]
for a constant $c>0$. The computation of the constant $c$ must rely on a more subtle handling of the sum in \eqref{eq:HypergeometricSeries}.
\end{remark}

\begin{remark}
The following example shows that the regularity conditions in Theorem \ref{theorem:OptimalApproxSkorohod} are less sensitive. Let 
\[
I=\int_{0}^{1} g(s)|W_t| dW_s = \int_{0}^{1} g(s)dW_s \diamond |W_t|
\]
for some $g \in C^1([0,1])$ and $t \in \frac{1}{m}\N$ fixed. Then, by Proposition \ref{prop:SkorohodAndWick}, $n \in m\N$ and \eqref{eq:IBPWick}, it is
\[
\ex[I|\mathcal{P}_n] = \left(g(1)W_1 - \sum_{i=1}^{n} \int_{(i-1)/n}^{i/n} g'(s) W_s^{\lin} ds\right) \diamond |W_t|
\]
and thus, making use of the arguments in the proof of Theorem \ref{theorem:OptimalApproxSkorohod}, which carry over to the Wiener chaos decomposition in \eqref{eq:absWtWienerChaos}, we have
\begin{align*}
e_n^2 &= \sum_{i,j=1}^{n} \int_{(i-1)/n}^{i/n} \int_{(j-1)/n}^{j/n} g'(s) g'(u) \ex[(B^n_s \diamond |W_t|)(B^n_u \diamond |W_t|)] ds \,du\\
&= \sum_{i=1}^{n} \int_{(i-1)/n}^{i/n} \int_{(i-1)/n}^{i/n} \ex[B^n_s B^n_u] ds \,du \,g'((i-1)/n)^2 \ex[|W_t|^2] +\Oo(n^{-3})\\
&= \frac{1}{12n^2} \left(\int_{0}^{1} g'(s)^2 ds\right) \ex[|W_t|^2] +\Oo(n^{-3}).
\end{align*}
We notice that the constant $\int_{0}^{1} g'(s)^2 ds\, \ex[|W_t|^2]$ equals $\int_{0}^{1}\ex[\mathcal{L} f(s, t,W_s,W_{t})^2] ds$ for $f(s,t,W_s,W_t) = g(s)|W_t|$ in Theorem \ref{theorem:OptimalApproxSkorohod} if we ignore the discontinuities. This indicates further extensions of Theorem \ref{theorem:OptimalApproxSkorohod} to weaker assumptions. Under the same assumptions, by analogous arguments, for 
\[
I=\int_{0}^{1} g(s)|W_t|\diamond W_s dW_s = \int_{0}^{1} g(s)W_s dW_s \diamond |W_t|,
\]
it is
\begin{align*}
e_n^2 &= \frac{1}{12n^2} \left(\int_{0}^{1} g'(s)^2 \ex\left[\left(|W_t|\diamond W_s\right)^2\right] ds\right) +\Oo(n^{-3}).
\end{align*}
We notice that $|W_t| \diamond W_1^{\diamond 2} \notin L^2(\Omega)$ due to \eqref{eq:absWtWienerChaos} and \eqref{eq:StirlingFormula}. Hence for the next integrand of this type, $(g(s) |W_t| \diamond W_s^{\diamond 2})_{s \in [0,1]}$, the Skorohod integral does not exist.
\end{remark}

\begin{remark}
Inspired by the Wiener chaos decomposition in \eqref{eq:absWtWienerChaosApprox},
for a fixed $T \in (0,1) \setminus \Q$, we define the random variable
\[
X_q = \sum\limits_{m \geq 0} \frac{1}{(2m)^{1+q} (2m)!!}\, W_{\{mT\}}^{\diamond 2m}.
\]
One can easily check by \eqref{eq:SeriesBinom} that $X_q \in L^2(\Omega)$ for all $q \in (-1/2,1/2)$ and the Skorohod integral 
\[
I_q = \int_{0}^{1} X_q dW_s = X_q \diamond W_1 
\]
exists in $L^2$ with
\[
\ex[I_q^2] \leq \sum\limits_{m \geq 0} \frac{(2m+1)!}{(2m)^{2+2q} (2m)!!^2} \{mT\}^{2m} <\infty.
\]
Due to Proposition \ref{prop:Weyl} there is no finite extension of $\P_n$ such that $I_q$ is exactly simulated. We sketch the optimal approximation results since the proofs can be done following the lines in the proof of Proposition \ref{prop:OptimalApproxAbsoluteValue}. Similarly to \eqref{eq:OptApproxAbsValue1}, it is  
\begin{align}
e_n^2 &= \sum\limits_{m \geq 1} \frac{1}{(2m)^{2+2q} (2m)!!^2}\ex\left[\left((W_{\{mT\}}^{\diamond 2m} -(W_{\{mT\}}^{\lin})^{\diamond 2m})\diamond W_1\right)^2\right].
\label{eq:IrregularMSE}
\end{align}
Due to fractional calculus on \eqref{eq:SeriesBinom}, we conjecture for the contained sum
\[
\sum\limits_{m \geq 0} (2m)^{-2q}\binom{2m}{m} \{mT\}^{2m} \asymp (1-\{mT\}^{2})^{-1/2-2q}.
\]
Following the proof of  Proposition \ref{prop:OptimalApproxAbsoluteValue}, we assume with the the sequence from Proposition \ref{prop:Weyl} for all $q \in (-1/2,1/2)$ the asymptotic behaviour
\[
e_{n_k}^2 \asymp n_k^{-1/2+2q}.
\]
\end{remark}

\begin{remark}
Simpler Skorohod integrals such that no finite extension of $\P_n$ is sufficient for exact approximation can be constructed as follows: Let $g : [0,1] \rightarrow [0,1]$ be a function with $g([0,1]) \notin \Q$, $|g([0,1])|=\infty$. Then the following Skorohod integral exists (is an element in the second Wiener chaos) and fulfills
$\int_{0}^{1} W_{g(s)} dW_s  = \sum_{j \in J} a_j W_{s_j} \diamond W_{t_j}$
for an infinite index set $J$. Such examples let us assume that Skorohod integrals of the following type
\[
\int_{0}^{1} f(W_{g(s)}) dW_s 
\]
for sufficiently irregular or discontinuous function $g: [0,1] \rightarrow [0,1]$ and appropriate $f \in C^2(\R)$, do not allow a finite extension of $\P_n$ with exact approximation and the optimal approximation behaves according to the rather irregular case in Theorem \ref{thm:OptApproxIrrational}.
\end{remark}


\begin{thebibliography}{plain}

\bibitem{Aigner} Aigner, M.
\textit{A course in enumeration} Graduate Texts in Mathematics, 238. Springer, Berlin, 2007.

\bibitem{Avram_Taqqu} F. Avram and M. Taqqu, Noncentral limit theorems and Appell polynomials, \textit{Ann. Probab.} \textbf{15} (2) (1987) 767--775. 


\bibitem{Buckdahn_Nualart} Buckdahn, R. and Nualart, D. Linear stochastic differential equations and Wick products. \textit{Probab. Theory Related Fields} \textbf{99} (4) (1994) 501--526.


\bibitem{Di_Nunno_und_so} Di Nunno, G. and {\O}ksendal, B. and Proske, F. \textit{Malliavin calculus for L\'{e}vy processes with applications to finance} Universitext. Springer, Berlin, 2009.

\bibitem{Fournie} Fourni\'e, E. and Lasry, J.-M. and Lebuchoux, J. and Lions, P.-L. and Touzi, N.
Applications of Malliavin calculus to Monte Carlo methods in finance. 
\textit{Finance Stoch.} \textbf{3} (4) (1999) 391--412.



\bibitem{Holden_Buch} Holden H. and {\O}ksendal, B. and Ub{\o}e, J. and Zhang, T.
\textit{Stochastic Partial Differential Equations. A Modeling, White Noise
    Functional Approach. Second Edition} Springer, New York, 2010.



\bibitem{Jentzen_MG_Y} Jentzen, A. and M\"uller-Gronbach, T. and Yaroslavtseva, L.
On stochastic differential equations with arbitrary slow convergence rates for strong approximation. \textit{Commun. Math. Sci.} \textbf{14} (6) (2016) 1477--1500.


\bibitem{Janson} Janson, S.
\textit{Gaussian Hilbert Spaces.}, Cambridge: Cambridge University Press, 1997.


\bibitem{Karatzas_Shreve} Karatzas, I. and Shreve, S. E.
\textit{Brownian motion and stochastic calculus.},
Second edition. Graduate Texts in Mathematics, 113. Springer. New York, 1991.

\bibitem{Kloeden_Platen}
Kloeden, P. and Platen, E. \textit{Numerical solution of stochastic differential equations.}
Applications of Mathematics, 23. Springer-Verlag, Berlin, 1992.


\bibitem{Kuo} Kuo, H.-H.
\textit{White Noise Distribution Theory.}
Probability and Stochastics Series. Boca Raton, FL: CRC Press, 1996.


\bibitem{Mueller_Gronbach} M\"uller-Gronbach, T.
Optimal pointwise approximation of SDEs based on Brownian motion at discrete points. \textit{Ann. Appl. Probab.} \textbf{14}, (4) (2004) 1605--1642.


\bibitem{NP} Neuenkirch, A. and Parczewski, P. Optimal approximation of Skorohod integrals. Accepted to \textit{J. Theoret. Probab.} (2016).


\bibitem{Nualart} Nualart, D.
\textit{The Malliavin Calculus and Related Topics.} Second Edition. Probability and its Applications. Springer, New York, 2006. 


\bibitem{P2} Parczewski, P. A Wick functional limit theorem. \textit{Probab. Math. Statist.} \textbf{34} (1) (2014) 127--145.



 \bibitem{Przybylowicz} Przyby\l owicz, P.
Optimal sampling design for approximation of stochastic It\^o integrals with application to the nonlinear Lebesgue integration. 
\textit{J. Comput. Appl. Math.} \textbf{245} (2013) 10--29. 


\bibitem{Weyl} Weyl, H. \"Uber die Gleichverteilung von Zahlen mod. Eins, \textit{Math. Ann.} \textbf{77}, (3) (1916) 313--352.

\end{thebibliography}
\end{document}